\documentclass[12pt,a4paper]{article}
\usepackage{amsmath,amssymb,amsthm}
\usepackage{hyperref}
\usepackage{enumitem}
\usepackage{mathtools}
\usepackage{adjustbox}
\usepackage[numbers]{natbib}
\usepackage{url}
\usepackage{xcolor}

\newtheorem{theorem}{Theorem}[section]

\theoremstyle{definition}

\newtheorem{assum}{Assumption}[section]

\numberwithin{equation}{section}
\hypersetup{colorlinks=true, linkcolor=blue, citecolor=blue}

    \makeatletter
\def\namedlabel#1#2{\begingroup
	#2%
	\def\@currentlabel{#2}%
	\phantomsection\label{#1}\endgroup
}
\makeatother

\allowdisplaybreaks

\begin{document}

\title{Lie-Bracket Nash Equilibrium Seeking with Bounded Update Rates for Noncooperative Games}

\author{Victor Hugo Pereira Rodrigues$^{a}$, \ Tiago Roux Oliveira$^{a}$, \\ Miroslav Krsti{\'c}$^{b}$, \ Tamer Ba{\c s}ar$^{c}$ \\ \\
	{\small{ $^{a}$State University of Rio de Janeiro,}}\\ 
	{\small{Department of Electronics and Telecommunication Engineering, Rio de Janeiro -- RJ, Brazil.}}\\ 
	{\small{\textcolor{black}{E-mail: victor.rodrigues@uerj.br, tiagoroux@uerj.br}}}\ \\ \\
	{\small{\textcolor{black}{$^{b}$University of California San Diego,}}}\\ 
	{\small{\textcolor{black}{Department of Mechanical and Aerospace Engineering, La Jolla -- CA, USA.}}}\\
	{\small{\textcolor{black}{E-mail: krstic@ucsd.edu}}}\ \\ \\
	{\small{\textcolor{black}{$^{c}$University of Illinois Urbana-Champaign,}}}\\ 
	{\small{\textcolor{black}{Department of Electrical and Computer Engineering, Urbana -- IL, USA.}}}\\
	{\small{\textcolor{black}{E-mail: basar1@illinois.edu}}}
	}

\date{}

\maketitle

\begin{abstract}
This paper proposes a novel approach for local convergence to Nash equilibrium in quadratic noncooperative games based on a distributed Lie-bracket extremum seeking control scheme. This is the first instance of noncooperative games being tackled in a model-free fashion integrated with the extremum seeking method of bounded update rates. In particular, the stability analysis is carried out using Lie-bracket approximation and Lyapunov's direct method. We quantify the size of the ultimate small residual sets around the Nash equilibrium and illustrate the theoretical results numerically on an example in an oligopoly setting. 
\end{abstract}

\section{Introduction}
\label{sec:introduction}
Game theory examines decision-making processes involving multiple individuals. The participants, often called players, agents, decision-makers, or simply persons, do not always have complete control over the outcome. Games are classified as \textit{noncooperative} when each participant pursues their own interests, which may partially conflict with the interests of others. Noncooperative game theory specifically addresses scenarios where players act independently without forming coalitions or collaborating with one another, unlike cooperative game theory. 
In such situations, the outcome is not determined by a single individual but requires a series of interdependent decisions made by multiple players. When these individuals assign different values to potential outcomes, the foundation for conflict is established \cite{Basar:1999}. 

A central concept in noncooperative games is the \textit{Nash equilibrium}, a solution where no player can improve their outcome by unilaterally changing their strategy. This equilibrium reflects a state of mutual best responses and is often referred to as a Nash solution. Unlike optimal control problems involving a single player, where optimality is well-defined and unambiguous, optimality in multi-person decision-making is inherently more complex and subjective. The Nash equilibrium provides one specific interpretation of optimality, as it balances competing interests in a stable and predictable manner \cite{N:1951}.

In this context, actions generally represent specific choices or controls, while strategies are decision rules informed by available information. A strategy is designed before uncertainties are resolved and dictates how a player will respond to future unknowns. Once the uncertainties are clarified, the resulting course of action is termed an action. In \cite{FKB:2012}, the authors study the problem of computing, in real time, the Nash equilibria of static noncooperative games with $N$ players by employing a non-model based approach. By utilizing \textit{extremum seeking} (ES) \cite{S:2024} with sinusoidal perturbations \cite{KW:2000}, the players achieve stable, local attainment of their Nash strategies without the need for any model information. Thereafter, the expansion of the algorithm proposed in \cite{FKB:2012} to different domains was given for delayed and PDE systems \cite{JOTA:2020,CSm:2024,SIAM:2022}, fixed-time convergence \cite{Poveda:2023,VSS:2024}, games with deception \cite{Tang:2024}, and event-triggered architectures \cite{CDC:2024}.

In this paper, we advance such designs and analyses based on classical ES \cite{KW:2000} to Nash equilibrium seeking (NES) with \textit{bounded update rates} \cite{SK_SCL:2014} for quadratic noncooperative games. We introduce a new NES scheme, in
which the uncertainty (payoff functions) is confined to the argument of a sine/cosine
function, thereby  guaranteeing known bounds on update rates and control efforts. In order to prove our stability results we employ Lie-bracket approximation/averaging \cite{DSEJ:2013} with highly oscillatory terms. \textcolor{black}{Standard Bogolyubov averaging from \cite{K:2002} does not suffice for the proposed bounded update law.}  Multi-agent ES setups which use similar game-theoretic approaches can be found in \cite{DSEJ:2013}. However, unlike the Lie-bracket NES approach proposed here, the authors in \cite{DSEJ:2013} proved local stability of multi-agent systems using ES feedback with \textit{unbounded} update rates.

We show that the trajectories of the proposed NES system can be approximated by the trajectories of a system which involves certain Lie brackets of the vector fields of the NES system. It turns out that the Lie-bracket system directly reveals the optimizing behavior of the corresponding NES system. Each player employing the proposed distributed algorithm is maximizing its own payoff, irrespective of what the other players' actions are. It is proven that if all the players are employing the Lie-bracket NES algorithms, they collectively converge to a small neighborhood of the Nash equilibrium. Hence, each of the players finds its optimal strategy, in an online fashion, even though they do not know the analytical forms of the payoff functions (neither the other players' nor their own) and neither have access to actions applied by the other players nor to payoffs achieved by the other players. Simulations illustrate the effectiveness of our approach on a numerical oligopoly market example. 

The paper is organized as follows. Section~\ref{sec2} introduces the standard assumptions (such as, diagonal dominance for the interaction matrix of the agents) for noncooperative games with $N$ players and quadratic payoff functions. While the proposed NES taking into account the design based on bounded update rates is shown in Section~\ref{sec3}, we consider the Lie-bracket approximation approach for NES in noncooperative games in Section~\ref{sec4}. The stability analysis with the proof of the main theorem is conducted in Section~\ref{sec5}. We illustrate the theoretical results numerically on an example in a 4-player setting, given in Section~\ref{sec6}. Concluding remarks are included in Section~\ref{sec7}.  
The theorem for Lie-bracket approximation \cite{DSEJ:2013} used in the paper is reviewed in the Appendix.

\section{Problem Formulation} \label{sec2}

Here, we consider the general class of static noncooperative games with players aiming  to maximize their quadratic payoff functions. 
General quadratic payoff functions are assumed. In particular, the payoff function for each player is expressed in the following form
\begin{align}
y_{i}=J_{i}(\theta(t)) = \frac{1}{2} \sum_{j=1}^N \sum_{k=1}^N H_{jk}^{i} \theta_j(t) \theta_k(t) + \sum_{j=1}^N h_{j}^{i} \theta_j(t) + c^{i}\,, \label{eq:Ji_v1}
\end{align}
where $\theta_j \in \mathbb{R}$ is the action of player $j$, $H_{ij}^{i}$, $h_{j}^{i}$, and $c^{i}$ are constants, $H_{ii}^{i}<0$, and $H_{kj}^{i} = H_{jk}^{i}$. Quadratic cost functions hold particular importance in game theory for two main reasons. First, they provide a useful second-order approximation to more complex nonlinear cost functions, enabling their application to a wide range of problems. Second, their mathematical simplicity allows for the derivation of closed-form equilibrium solutions, which facilitate deeper understanding of the behavior and properties of equilibrium concepts in various scenarios \cite{Basar:1999}.

In order to determine the Nash equilibrium solution in
strictly concave quadratic games, we differentiate (\ref{eq:Ji_v1}) with respect to $\theta \in \mathbb{R}^{N}$, set the resulting derivatives to zero, and solve the resulting system of equations. In other words, for a quadratic game with payoff functions as (\ref{eq:Ji_v1}), the $N$-player game admits a Nash equilibrium $\theta^{\ast} = [\theta_1^{\ast}\,, \dots\,, \theta_N^{\ast}]^{\top}$ if and only if
\begin{align}
\sum_{j=1}^N H_{ij}^{i} \theta_j^{\ast} + h_{i}^{i} = 0, \quad \forall i \in \{1, \dots, N\}\,. \label{eq:NEcond_v1}
\end{align}
In the matrix form, (\ref{eq:NEcond_v1}) is expressed as
\begin{align}
H\theta^{\ast}+h = 0\,, \label{eq:NEcond_v2}
\end{align}
where
\begin{align}
H = \begin{bmatrix}
H_{11}^{1} & H_{12}^{1} & \cdots & H_{1N}^{1} \\
H_{21}^{2} & H_{22}^{2} & \cdots & H_{2N}^{2} \\
\vdots & \vdots & \ddots & \vdots \\
H_{N1}^{N} & H_{N2}^{N} & \cdots & H_{NN}^{N}
\end{bmatrix}, ~h = \begin{bmatrix}
h_{1}^{1} \\
h_{2}^{2} \\
\vdots \\
h_{N}^{N}
\end{bmatrix},
 ~
\theta^{\ast} = \begin{bmatrix}
\theta_{1}^{\ast} \\
\theta_{2}^{\ast} \\
\vdots \\
\theta_{N}^{\ast}
\end{bmatrix}. \label{eq:NEcond_v3}
\end{align}

The quadratic $N$-player nonzero-sum static game, characterized by the payoff functions in equation (\ref{eq:Ji_v1}), possesses a unique Nash equilibrium if the matrix $H$, defined in equation (\ref{eq:NEcond_v3}), is invertible.

This brings us to the following assumptions.

\begin{assum} \label{assum:Nash}
 The \textit{unknown} Nash equilibrium is given by
\begin{align}
\theta^{\ast} = - H^{-1} h \,. \label{eq:NE}
\end{align}
\end{assum}
%
\begin{assum} \label{assum:SDD} The \textit{unknown }matrix $H$ is strictly diagonally dominant \cite{FKB:2012}, {\it i.e.},
\begin{align}
\sum_{j \neq i} |H_{ij}^{i}| < |H_{ii}^{i}|, \quad \forall i \in \{1, \dots, N\}. 
\end{align}
\end{assum}

Under Assumption~\ref{assum:SDD}, $H$ is invertible and the Nash equilibrium $\theta^{\ast}$ exists and is unique.

\section{Nash Equilibrium Seeking with Bounded Update Rates} \label{sec3}

\begin{figure*}[h!]
	\centering 
         \resizebox{.8\linewidth}{!}{\includegraphics{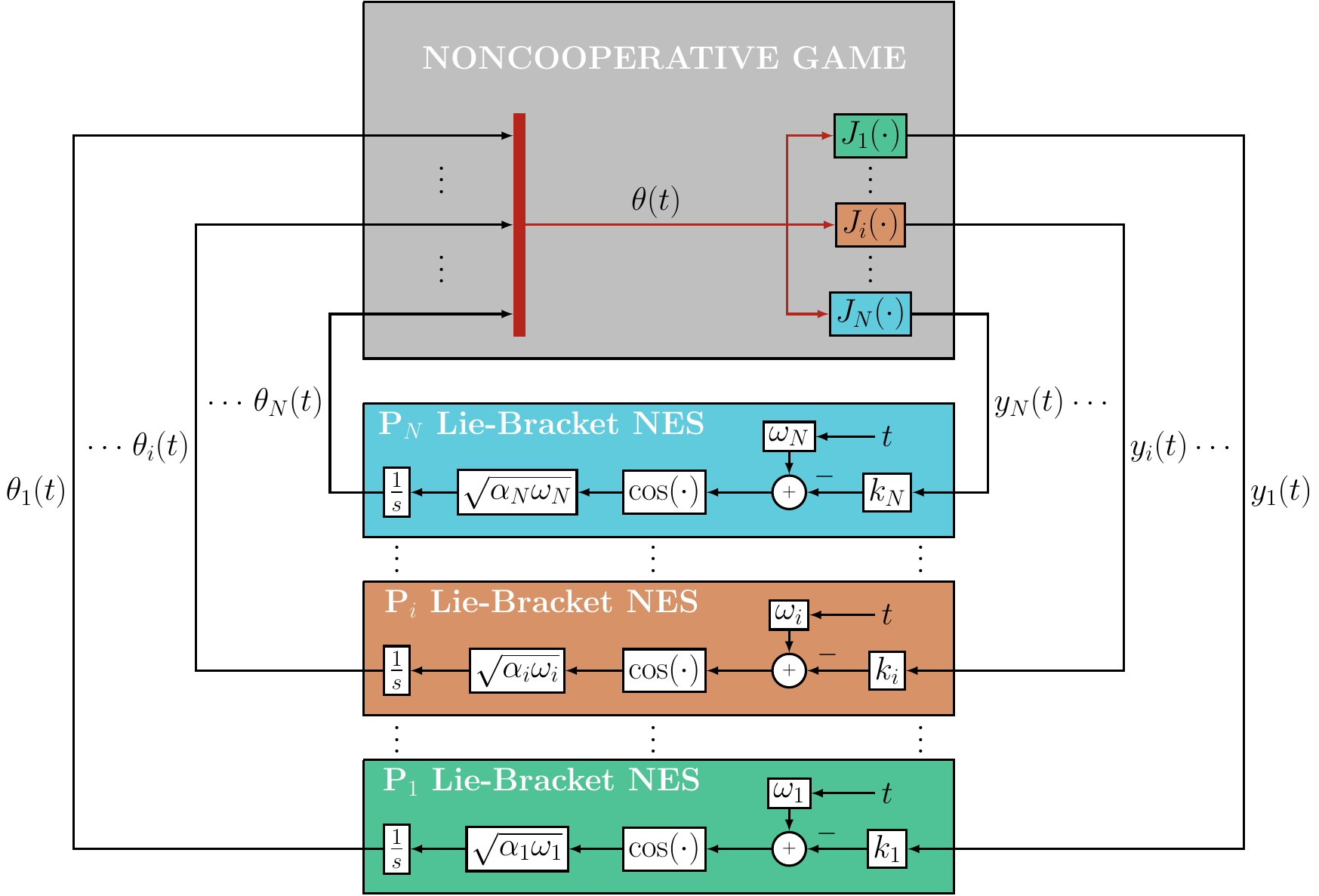}} 
        \caption{Block Diagram of the proposed Lie-bracket Nash equilibrium seeking with bounded update rates.}
        \label{fig:LB-NES_BD}
\end{figure*}

In Figure~\ref{fig:LB-NES_BD}, the Player-$i$ actions are set according to the following time-varying strategy
\begin{align}
\dot{\theta}_{i}(t)&=\sqrt{\alpha_{i}\omega_{i}}\cos(\omega_{i}t-k_{i}J_{i}(\theta))\,, \label{eq:dthetai/dt_v1}
\end{align}
where the frequency is $\omega_{i}$, and $\alpha_{i}$, $k_{i}$ are positive controller gains. \textcolor{black}{The update law (\ref{eq:dthetai/dt_v1}) was inspired by the ES strategy with bounded update rates, as presented in \cite{SK_SCL:2014,SK:2017}. Bounded
ES was introduced to optimize nonlinear maps embedded within oscillatory dynamics, enabling the design of smooth control actions with guaranteed constraints on update rates and control efforts \cite{SK_SCL:2014,S:2024}. This approach has been further extended to include discontinuous dithers \cite{SS:2016}, and to dynamic maps through the use of singularly perturbed Lie-bracket approximations \cite{DKSE:2015}. Furthermore, bounded ES has demonstrated remarkable success in optimizing particle accelerators performance \cite{SHT:2021,SHVGPDK:2020,SBTKZSSBD:2019}.}

\begin{assum} \label{assum:w}
The probing frequencies $\omega_{i}$, for all $i \in \{1\,, \ldots\,,N \}$ satisfy
\begin{align}
\omega_{i} = a_{i} \omega \quad \mbox{and} \quad a_{i} \neq a_j\,, \quad a_{i} \in \mathbb{Q}_{++} \,,
\end{align}
for all $\omega \in (0,\infty)$ and $i,j \in \{1\,, \ldots\,,N \}$, where $\mathbb{Q}_{++}$ denotes the set of positive rational numbers..
\end{assum}

Assumption~\ref{assum:w} regarding the parameter $\omega$ ensures that certain terms vanish in the corresponding Lie-bracket approximation of Section~\ref{sec4}.  For more details, see \cite{DSEJ:2013}.

By using some trigonometric identities,  (\ref{eq:dthetai/dt_v1}) is rewritten as
\begin{align}
\dot{\theta}_{i}(t)&=\sqrt{\alpha_{i}\omega_{i}}\cos(k_{i}J_{i}(\theta(t)))\cos(\omega_{i}t)+\nonumber \\
&\quad+\sqrt{\alpha_{i}\omega_{i}}\sin(k_{i}J_{i}(\theta(t)))\sin(\omega_{i}t)\,. \label{eq:dthetai/dt_v2}
\end{align}
By treating $\sin(\omega_{i} t)$ and $\cos(\omega_{i} t)$ as auxiliary inputs, {\it i.e.}, $u_{1}^{i}(\omega_{i} t) := \sin(\omega_{i} t)$ and $u_{2}^{i}(\omega_{i} t) := \cos(\omega_{i} t)$, we can represent the dynamics in (\ref{eq:dthetai/dt_v2}) into the following input-affine form
\begin{align}
\dot{\theta}_{i}(t)&=\sqrt{\alpha_{i}}\cos(k_{i}J_{i}(\theta(t)))\sqrt{\omega_{i}}u_{2}^{i}(\omega_{i} t)+\nonumber \\
&\quad+\sqrt{\alpha_{i}}\sin(k_{i}J_{i}(\theta(t)))\sqrt{\omega_{i}}u_1^{i}(\omega_{i} t)\,. \label{eq:dthetai/dt_v3}
\end{align}

To simplify the stability analysis, we concatenate all actions of the players into a single vector, $\theta(t) := [\theta_{1}(t)\,, \ldots\,, \theta_{i}(t)\,, \ldots\,,  \theta_{N}(t)]^\top$. This approach allows us to treat the game dynamics as a unified system while preserving the independence of each player's strategy. Thus, the estimation error of the Nash equilibrium obeys the following dynamics
\begin{align}
\dot{\theta}(t)&=\sum_{j=1}^{N}(b_{1}^{j}(t,\theta)\sqrt{\omega_{j}}u_{1}^{j}(\omega_{j} t)+b_{2}^{j}(t,\theta)\sqrt{\omega_{j}}u_2^{j}(\omega_{j} t))\,, \label{eq:dthetaA/dt_v1}
\end{align}
with $b_{1}^{j}(t,\tilde{\theta})$ and $b_{2}^{j}(t,\tilde{\theta})$ having non-zero entries only at positions corresponding to agent $i$ and zeros elsewhere, {\it i.e.}, $$b_{1}^{j}(t,\tilde{\theta}) = [0\,, \dots\,, 0\,, \sqrt{\alpha_{j}}\sin(k_{j}J_{j}(\theta^{\ast}+\tilde{\theta}(t))), 0\,, \ldots\,, 0]^\top $$ and $$b_{2}^{j}(t,\tilde{\theta}) = [0\,, \dots\,, 0\,, \sqrt{\alpha_{j}}\cos(k_{j}J_{j}(\theta^{\ast}+\tilde{\theta}(t)))\,, 0\,, \ldots \,, 0]^\top.$$ Due to Assumption~\ref{assum:w}, we have that $a_i$ can be written as $a_i = \frac{p_i}{q_i}$, with $p_i, q_i \in \mathbb{N}$ and define $q := \prod_{i=1}^N q_i$ and 
%
\begin{equation} \label{tildeomega_golden}
\tilde{\omega} := \frac{\omega}{q}\,.
\end{equation}
Thus, $a_i \omega = p_i \prod_{j \neq i} q_j \tilde{\omega} = n^i \tilde{\omega}$, for $i = 1, \dots, N$, $j = 1, 2$, and $n^{i'} \!:=\! p_{i'} \prod_{j \neq i'} q_j \!\in\! \mathbb{N}$. Hence, we can rewrite (\ref{eq:dthetaA/dt_v1}) as
{\small
\begin{align}
\dot{\theta}(t)&=\sum_{j=1}^{N}(\sqrt{n_{j}}b_{1}^{j}(t,\theta)\sqrt{\tilde{\omega}}u_{1}^{j}(n_{j}\tilde{\omega} t)+\sqrt{n_{j}}b_{2}^{j}(t,\theta)\sqrt{\tilde{\omega}}u_2^{j}(n_{j}\tilde{\omega} t))\,. \label{eq:dthetaA/dt_v2}
\end{align}
}
$\!\!\!$Since the auxiliary inputs $u_{1}^{j}(n_{j}\tilde{\omega} t)$ and $u_{2}^{j}(n_{j}\tilde{\omega} t)$ appear linearly in the dynamics (\ref{eq:dthetaA/dt_v2}), influencing the system through a direct linear combination with the vector fields $b_{1}^{j}(\tilde{\theta})$ and $b_{2}^{j}(\tilde{\theta})$, respectively, and, in addition, the fast oscillatory nature of $u_{1}^{j}(n_{j}\tilde{\omega} t)$ and $u_{2}^{j}(n_{j}\tilde{\omega} t)$ introduces non-integrable motion constraints that prevent the dynamics from being expressed as simple algebraic or integrable relationships between $\tilde{\theta}$ and $t$, thus, the dynamics in (\ref{eq:dthetaA/dt_v2}) is classified as affine and \textit{nonholonomic} \cite{GL:1993} with respect to the inputs $u_{1}^{j}(n_{j}\tilde{\omega} t)$ and $u_{2}^{j}(n_{j}\tilde{\omega} t)$, for all $j = \{1\,, \ldots \,, N\}$. 

\section{Lie-Bracket Approximations} \label{sec4}

Now, Lie-bracket approximation technique is employed to analyze the asymptotic behavior of the nonholonomic dynamics (\ref{eq:dthetaA/dt_v2}) subjected to highly oscillatory inputs $u_{1}^{j}(n_{j}\tilde{\omega} t)$ and $u_{2}^{j}(n_{j}\tilde{\omega} t)$, with $j = \{1\,, \ldots \,, N\}$. The variables that depend on $ \tilde{\omega} t$ are identified as the fast variables, while those dependent of $\tilde{\theta}$ and $t$ are treated as the slow variables. From \cite{DSEJ:2013}, it is possible to simplify the system for sufficiently large $\tilde{\omega}$ by deriving an ``average'' system that effectively eliminates the fast variable and provides a reliable approximation of the original system described by (\ref{eq:dthetaA/dt_v2}). We describes this Lie-bracket approximation of (\ref{eq:dthetaA/dt_v2}) by means of the next Lie-bracket NES equations: 
\begin{align}
\dot{\bar{\theta}}(t)&=\sum_{j=1}^{N}[\sqrt{n_{j}}b_{1}^{j},\sqrt{n_{j}}b_{2}^{j}](t,\bar{\theta})\nu_{kl}^{ji}(t)\,, \quad \bar{\theta}(0)=\theta(0)\,, \label{eq:dtheta-/dt_v1} \\
\nu_{kl}^{ji}(t)&=\frac{1}{T}\int_{0}^{T}u_{k}^{i}(n_{i}\beta)\int_{0}^{\beta}u_{l}^{j}(n_{j}\alpha)d\alpha d\beta\,. \label{eq:nu_v1} 
\end{align}
The average dynamics (\ref{eq:dtheta-/dt_v1}) involves Lie brackets $[b_{1}^{j}\,,b_{2}^{j}]$, which represent the commutators of vector fields associated with the auxiliary control inputs. Lie brackets indicate that the system's controllable motion is not confined to the original directions of the vector fields $b_{1}^{j}$ and $b_{2}^{j}$ but also includes motion generated by the interaction of these fields.  The variable\linebreak $\bar\theta$ generated by (\ref{eq:dtheta-/dt_v1}), along with (\ref{eq:nu_v1}), also referred to as the Lie-bracket system, approximates the behavior of (\ref{eq:dthetaA/dt_v2}). 

Since $u_{k}^{j}(n_{j}\Theta)$ in (\ref{eq:dthetaA/dt_v2}) belongs to the set $\{\sin(n_{j} \Theta), \cos(n_{j} \Theta)\}$, they are $T$-periodic functions in terms of $n_{j} \Theta$, with $j = \{1, \ldots, N\}$, $k = \{1, 2\}$, and $n_{j} \in \mathbb{N}$, {\it i.e.}, $T=2\pi$. Thus, the parameters $\nu_{kl}^{ji}(t)$ in the following are constant for all $i, j = 1, . . . , N$ and $k, l = 1, 2$, such that 
\begin{align}
\nu_{kl}^{ji}(t)&=\nu_{kl}^{ji}= \displaystyle\begin{cases} -\frac{1}{2n_{j}}\,, \quad \mbox{ for all } n_{i}=n_{j} \mbox{ and } k=l\,, \\ ~~~0~~\,, \quad \mbox{ otherwise } \end{cases}\,. 
\end{align}
Therefore, (\ref{eq:dtheta-/dt_v1}) is rewritten as
\begin{align}
\dot{\bar{\theta}}(t)&=\sum_{j=1}^{N}[\sqrt{n_{j}}b_{1}^{j},\sqrt{n_{j}}b_{2}^{j}](t,\bar{\theta})\left(-\frac{1}{2n_{j}}\right)\nonumber \\
&=-\frac{1}{2}\sum_{j=1}^{N}[b_{1}^{j},b_{2}^{j}](t,\bar{\theta}) \,, \quad \bar{\theta}(0)=\theta(0) \,. \label{eq:dtheta-/dt_v3}
\end{align}
On the other hand, the Lie-bracket between the vector fields $b_{1}^{j}(t,\bar{\theta})$ and $b_{2}^{j}(t,\bar{\theta})$, defined by $[b_{1}^{j},b_{2}^{j}](t,\bar{\theta})= \dfrac{\partial b_{2}^{j}(t,\bar{\theta})}{\partial \bar{\theta}}b_{1}^{j}(t,\bar{\theta})-\dfrac{\partial b_{1}^{j}(t,\bar{\theta})}{\partial \bar{\theta}}b_{2}^{j}(t,\bar{\theta})$, with 
\begin{small}
\begin{align}
\dfrac{\partial b_{1}^{j}(t,\bar{\theta})}{\partial \bar{\theta}} &= \begin{bmatrix}0 & \ldots & 0 & \ldots & 0 \\ \vdots & \ldots & \vdots & \ldots & \vdots \\ 0 & \ldots & 0 & \ldots & 0 \\ 0 & \ldots & \sqrt{\alpha_{j}}k_{j}\cos(k_{j}J_{j}(\bar{\theta}))\dfrac{\partial J_{j}(\bar{\theta})}{\partial \bar{\theta}_{j}} & \ldots & 0 \\ 0 & \ldots & 0 & \ldots & 0 \\ \vdots & \ldots & \vdots & \ldots & \vdots \\ 0 & \ldots & 0 & \ldots & 0 \end{bmatrix}\,, \\
\dfrac{\partial b_{2}^{j}(t,\bar{\theta})}{\partial \bar{\theta}} &= \begin{bmatrix}0 & \ldots & 0 & \ldots & 0 \\ \vdots & \ldots & \vdots & \ldots & \vdots \\ 0 & \ldots & 0 & \ldots & 0 \\ 0 & \ldots & -\sqrt{\alpha_{j}}k_{j}\sin(k_{j}J_{j}(\bar{\theta}))\dfrac{\partial J_{j}(\bar{\theta})}{\partial \bar{\theta}_{j}} & \ldots & 0 \\ 0 & \ldots & 0 & \ldots & 0 \\ \vdots & \ldots & \vdots & \ldots & \vdots \\ 0 & \ldots & 0 & \ldots & 0 \end{bmatrix} \,,
\end{align}
\end{small}
$\!\!\!$yields
\begin{small}
\begin{align} 
[b_{1}^{j},b_{2}^{j}](t,\bar{\theta})=-\left[0\,, \dots\,, 0\,, \alpha_{j}k_{j}\dfrac{\partial J_{j}(\bar{\theta})}{\partial \bar{\theta}_{j}}, 0\,, \ldots\,, 0\right]^\top \,. \label{eq:[b1,b2]}
\end{align}
\end{small}
$\!\!\!$By plugging (\ref{eq:[b1,b2]}) into (\ref{eq:dtheta-/dt_v3}), one has 
\begin{small}
\begin{align}
&\dot{\bar{\theta}}(t)=\frac{1}{2}\left[\alpha_{1}k_{1}\dfrac{\partial J_{1}(\bar{\theta})}{\partial \bar{\theta}_{1}}\,, \dots\,, \alpha_{i}k_{i}\dfrac{\partial J_{i}(\bar{\theta})}{\partial \bar{\theta}_{i}}\,, \ldots\,, \alpha_{N}k_{N}\dfrac{\partial J_{N}(\bar{\theta})}{\partial \bar{\theta}_{N}}\right]^\top\nonumber \\
&=\frac{1}{2}AK \nabla J(\bar{\theta})\,, \quad \bar{\theta}(0)=\theta(0)\,,  \label{eq:dtheta-/dt_v4} \\
&\nabla J(\bar{\theta}):=\left[\dfrac{\partial J_{1}(\bar{\theta})}{\partial \bar{\theta}_{1}}\,, \dots\,, \dfrac{\partial J_{i}(\bar{\theta})}{\partial \bar{\theta}_{i}}\,, \ldots\,, \dfrac{\partial J_{N}(\bar{\theta})}{\partial \bar{\theta}_{N}}\right]^\top\,, \label{eq:nablaJ_v1} \\
&A=
\begin{bmatrix} \alpha_{1} &      0     & \ldots &    0       \\
                    0      & \alpha_{2} & \ddots &  \vdots    \\ 
								 \vdots    &   \ddots   & \ddots &    0       \\
										0      &   \ldots   &   0    & \alpha_{N}
\end{bmatrix}	\,, \quad
K=
\begin{bmatrix}    k_{1}   &      0     & \ldots &    0       \\
                    0      &      k_{2} & \ddots &  \vdots    \\ 
								 \vdots    &   \ddots   & \ddots &    0       \\
										0      &   \ldots   &   0    &   k_{N}
\end{bmatrix}\,.										
\end{align}
\end{small}
$\!\!\!$We note that, for each differentiable payoff function $J_{i} : \mathbb{R}^{N} \mapsto \mathbb{R}$ in (\ref{eq:Ji_v1}), the gradient of $J_{i}$ at $(\theta_{1}\,, \ldots\,,\theta_{N})$ is the vector in the space given by $\nabla J_{i} : = \left[\frac{\partial J_{i}}{\partial \theta_1}\,, \ldots\,, \frac{\partial J_{i}}{\partial \theta_i} \,,\ldots \,, \frac{\partial J_{i}}{\partial \theta_N}\right]^{\top}$. \textcolor{black}{The gradients point into the direction in which the value of $J_{i}$ is  increasing the fastest and the direction that is orthogonal to the level surfaces of $J_{i}$ \cite{MT:2003}.} Nevertheless, in the context of noncooperative games, the {\it pseudo-gradient} (\ref{eq:nablaJ_v1}) is preferred over the true gradient because it effectively captures the interdependence of players' strategies and payoff functions, a critical aspect of such games. Unlike the gradient, which requires a single scalar potential function, the pseudo-gradient combines the individual gradients of each player's utility function into a unified framework, enabling the analysis and computation of the Nash equilibrium.

By calculating the partial derivative of $J_{i}$ in (\ref{eq:Ji_v1}) with respect to $\bar{\theta}_{i}$, one has
\begin{align}
\frac{\partial J_{i}(\bar{\theta}(t))}{\partial \bar{\theta}_{i}(t)}& = \sum_{j=1}^{N} H_{ij}^{i} \bar{\theta}_j(t) +  h_{i}^{i} \,, \label{eq:dJi/dthetai_v1}
\end{align}
thus, the {\it pseudo-gradient} (\ref{eq:nablaJ_v1}) can be written as
\begin{align}
\nabla J(\bar{\theta}):=H\bar{\theta}(t)+h\,, \label{eq:nablaJ_v2}
\end{align}  
with $H$ and $h$ given by (\ref{eq:NEcond_v3}). Therefore, the dynamics in (\ref{eq:dtheta-/dt_v4}) becomes 
\begin{align}
\dot{\bar{\theta}}(t)&=\frac{1}{2}AK (H\bar{\theta}(t)+h)\,, \quad \bar{\theta}(0)=\theta(0) \,. \label{eq:dtheta-/dt_v5} 									
\end{align}

The goal of the control strategy is to make $\tilde{\theta}(t)$ as small as possible, so that the Lie-bracket approximation $\bar{\theta}(t)$ is driven to the Nash equilibrium $\theta^{\ast}$ and the output $\theta(t)$ to a small neighborhood of this. 

\section{Stability Analysis} \label{sec5}

\begin{theorem}
For the $N$-player quadratic noncooperative game in (\ref{eq:Ji_v1}), consider the dynamics of the concatenated strategies employed by the players (\ref{eq:dthetaA/dt_v2}). The Lie-bracket approximation is given by (\ref{eq:dtheta-/dt_v5}). Under Assumptions~\ref{assum:Nash}, \ref{assum:SDD}, and \ref{assum:w}, for $\|\theta(0)-\theta^{\ast}\|$ sufficiently small and $\tilde{\omega}$ in (\ref{tildeomega_golden}) sufficiently large, there exist constants $M,m>0$ such that, for all $t\geq 0$: 
\begin{align}
\|\theta(t)-\theta^{\ast}\| \leq M \exp(-mt)\|\theta(0)-\theta^{\ast}\|+\mathcal{O}\left(\frac{1}{\tilde{\omega}}\right)\,, \label{eq:thm}
\end{align}
where $\|\theta(t)-\theta^{\ast}\|=\sqrt{(\theta(t)-\theta^{\ast})^{\top}(\theta(t)-\theta^{\ast})}$ is the Euclidean norm of the discrepancy between the players' actions and the Nash equilibrium. 
\end{theorem}

\begin{proof}
For the purpose of analysis, we define the estimation error between the players' actions and the Nash equilibrium as 
\begin{align}
\tilde{\theta}(t) = \bar{\theta}(t)-\theta^{\ast}\,, \label{eq:thetaA_v1}
\end{align}
whose time-derivative is given by 
\begin{align}
\dot{\tilde{\theta}}(t)&=\dot{\bar{\theta}}(t) \nonumber \\
&=\frac{1}{2}AK (H\bar{\theta}(t)+h) \nonumber \\
&=\frac{1}{2}AK (H\tilde{\theta}(t)+\underbrace{H\theta^{\ast}+h}_{\text{\scriptsize{ \mbox{$=0$ , from (\ref{eq:NEcond_v1})}}}}) \nonumber \\
&=\frac{1}{2}AK H\tilde{\theta}(t)\,. \label{eq:dthetaA/dt_v6} 									
\end{align}

The average error system (\ref{eq:dthetaA/dt_v6}) can be rewritten as
\begin{align}
\dot{\tilde{\theta}}(t)&=\mathcal{A}\tilde{\theta}(t)\,, \label{eq:dthetaA/dt_v7}
\quad \mathcal{A}=\begin{bmatrix}
\kappa_{1}H_{11}^{1} & \kappa_{1}H_{12}^{1} & \cdots & \kappa_{1}H_{1N}^{1} \\
\kappa_{2}H_{21}^{2} & \kappa_{2}H_{22}^{2} & \cdots & \kappa_{2}H_{2N}^{2} \\
\vdots & \vdots & \ddots & \vdots \\
\kappa_{N}H_{N1}^{N} & \kappa_{N}H_{N2}^{N} & \cdots & \kappa_{N}H_{NN}^{N}
\end{bmatrix} 	\,, 								
\end{align}
where $\kappa_{i}=\dfrac{\alpha_{i}k_{i}}{2}$ for all $i \in \{1, \ldots, N\}$. From the Gershgorin Circle Theorem \cite[Theorem 6.1.1]{HJ:1985}, we have 
$\lambda(\mathcal{A}) \subseteq \bigcup_{i=1}^N \rho_i$, where $\lambda(\mathcal{A})$ denotes the spectrum of $\mathcal{A}$, and $\rho_i$ is a Gershgorin disc:
\begin{align}
\rho_i = \frac{\alpha_{i}k_{i}}{2} \left\{ z \in \mathbb{C} \ \middle| \ |z - H_{ii}^i| < \sum_{j \neq i} |H_{ij}^i| \right\}\,.
\end{align}
Since $H_{ii}^i < 0$ and $H$ is strictly diagonally dominant, the union of the Gershgorin discs lies strictly in the left half of the complex plane, and we conclude that $\text{Re}\{\lambda\} < 0$, for all $\lambda \in \lambda(\mathcal{A})$. Thus, given any matrix $Q = Q^\top > 0$, there exists a matrix $P = P^\top > 0$ satisfying the Lyapunov equation $P\mathcal{A} + \mathcal{A}^\top P = -Q$.

Now, consider the following candidate Lyapunov function for the error system (\ref{eq:dthetaA/dt_v7})
\begin{align}
V(t)=\tilde{\theta}^{\top}(t)P\tilde{\theta}(t) \,,\, P=P^T>0, \label{eq:V_v1}
\end{align}  
with time-derivative satisfying 
\begin{align}
\dot{V}(t)&=\tilde{\theta}^{\top}(t)\mathcal{A}^{\top}P\tilde{\theta}(t)+\tilde{\theta}^{\top}(t)P\mathcal{A}\tilde{\theta}(t) \nonumber \\
&=-\tilde{\theta}^{\top}(t)Q\tilde{\theta}(t) \nonumber \\
&\leq -\lambda_{\min}(Q)\|\tilde{\theta}(t)\|^2\,. \label{eq:dV_v1}
\end{align}
By using the Rayleigh-Ritz Inequality \cite{K:2002}, 
\begin{align}
\lambda_{\min}(P)\|\tilde{\theta}(t)\|^2\leq V(t) \leq \lambda_{\max}(P)\|\tilde{\theta}(t)\|^2\,, \label{eq:RR}
\end{align}
an upper bound for (\ref{eq:dV_v1}) is given by
\begin{align}
\dot{V}(t)&\leq -\frac{\lambda_{\min}(Q)}{\lambda_{\max}(P)}V(t)\,. \label{eq:dV_v2}
\end{align}
Then, invoking \cite[Comparison Lemma]{K:2002} an upper bound $\bar{V}(t)$ for $V(t)$,   
\begin{align}
V(t)\leq \bar{V}(t) \,, \quad \forall t \geq 0 \,, \label{eq:VV-_v1}
\end{align}
 is given by the solution of the equation
\begin{align}
\dot{\bar{V}}(t)&= -\frac{\lambda_{\min}(Q)}{\lambda_{\max}(P)}\bar{V}(t)\,, \quad \bar{V}(0)=V(0)\,.
\end{align}
In other words,
\begin{align}
\bar{V}(t)=\exp\left(-\frac{\lambda_{\min}(Q)}{\lambda_{\max}(P)}t\right)V(0)\,, \label{eq:V-_v1}
\end{align}
and inequality (\ref{eq:VV-_v1}) is rewritten as
\begin{align}
V(t)\leq \exp\left(-\frac{\lambda_{\min}(Q)}{\lambda_{\max}(P)}t\right)V(0) \,, \quad \forall t \geq 0 \,. \label{eq:VV-_v2}
\end{align}
Now, lower bounding the left-hand side and upper bounding the right-hand size of (\ref{eq:VV-_v2}), with the corresponding sides of (\ref{eq:RR}), one gets
\begin{align}
\lambda_{\min}(P)\|\tilde{\theta}(t)\|^2\leq \exp\left(-\frac{\lambda_{\min}(Q)}{\lambda_{\max}(P)}t\right)\lambda_{\max}(P)\|\tilde{\theta}(0)\|^2 \,, 
\end{align}
and, consequently,
\begin{align}
\|\tilde{\theta}(t)\|\leq \sqrt{\frac{\lambda_{\max}(P)}{\lambda_{\min}(P)}}\exp\left(-\frac{\lambda_{\min}(Q)}{\lambda_{\max}(P)}t\right)\|\tilde{\theta}(0)\| \,, \quad \forall t \geq 0 \,.  \label{eq:VV-_v3}
\end{align}
By adding $\theta(t)$ to both sides of (\ref{eq:thetaA_v1}) such that $\tilde{\theta}(t)+\theta(t)=\bar{\theta}(t)-\theta^{\ast}+\theta(t)$ or, equivalently,  $\theta(t)-\theta^{\ast}=\tilde{\theta}(t)+\theta(t)-\bar{\theta}(t)$, one can write down the norm as: 
\begin{align}
\|\theta(t)-\theta^{\ast}\|=\|\tilde{\theta}(t)+\theta(t)-\bar{\theta}(t)\| \,, \quad \forall t \geq 0 \,.  \label{eq:normTheta_v1}
\end{align}
By using the Triangle inequality, one has $\|\tilde{\theta}(t)+\theta(t)-\bar{\theta}(t)\|\leq \|\tilde{\theta}(t)\|+\|\theta(t)-\bar{\theta}(t)\|$, which is upper bounded by  
\begin{align}
\|\theta(t)-\theta^{\ast}\|\leq \|\tilde{\theta}(t)\|+\|\theta(t)-\bar{\theta}(t)\| \,, \quad \forall t \geq 0 \,.  \label{eq:normTheta_v2}
\end{align}
Since (\ref{eq:dtheta-/dt_v5}) is a Lie-bracket approximation for the nonholonomic dynamics (\ref{eq:dthetaA/dt_v2}) with highly oscillatory inputs of multiple frequencies of $\tilde{\omega}$, by invoking  \cite[Theorem~2.1]{GL:1993}---see also the Appendix, one gets 
\begin{align}
\|\theta(t)-\bar{\theta}(t)\| \leq \mathcal{O}\left(\frac{1}{\tilde{\omega}}\right)\,, \quad \forall t \geq 0 \,,  \label{eq:normTheta_v3}
\end{align}     
where $\mathcal{O}\left(\frac{1}{\tilde{\omega}}\right)$ is a residual set which tends to zero as $\tilde{\omega} \to \infty$. 
Now, by using (\ref{eq:VV-_v3}) and (\ref{eq:normTheta_v3}), inequality (\ref{eq:normTheta_v2}) is upper bounded by
\begin{align}
\|\theta(t)-\theta^{\ast}\|&\leq \sqrt{\frac{\lambda_{\max}(P)}{\lambda_{\min}(P)}}\exp\left(-\frac{\lambda_{\min}(Q)}{\lambda_{\max}(P)}t\right)\|\tilde{\theta}(0)\| \nonumber \\
&\quad+\mathcal{O}\left(\frac{1}{\tilde{\omega}}\right) \,, \quad \forall t \geq 0 \,.  \label{eq:normTheta_v4}
\end{align}
From (\ref{eq:dtheta-/dt_v5}), by setting $\bar{\theta}(0)=\theta(0)$ and, therefore, from (\ref{eq:normTheta_v1}), $\|\tilde{\theta}(0)\|=\|\theta(0)-\theta^{\ast}\|$, we are able to rewrite (\ref{eq:normTheta_v4}) as 
\begin{align}
\|\theta(t)-\theta^{\ast}\|&\leq \sqrt{\frac{\lambda_{\max}(P)}{\lambda_{\min}(P)}}\exp\left(-\frac{\lambda_{\min}(Q)}{\lambda_{\max}(P)}t\right)\|\theta(0)-\theta^{\ast}\| \nonumber \\
&\quad+\mathcal{O}\left(\frac{1}{\tilde{\omega}}\right) \,, \quad \forall t \geq 0 \,.  \label{eq:normTheta_v5}
\end{align}
Finally, inequality (\ref{eq:normTheta_v5}) satisfies (\ref{eq:thm}) with constants
\begin{align}
M&=\sqrt{\frac{\lambda_{\max}(P)}{\lambda_{\min}(P)}} \,, \label{eq:M} \\
m&=\frac{\lambda_{\min}(Q)}{\lambda_{\max}(P)}\,, \label{eq:m}
\end{align}
which completes the proof.
\end{proof}

\section{Simulation Results} \label{sec6}

This section presents simulation results by illustrating the performance of the proposed Lie-bracket NES control scheme. The system under study is a static noncooperative game with four firms in an oligopoly market structure that compete to maximize their profits $J_{i}(t)$ by setting the price $u_{i}(t)$ of their product without sharing information among the players  (see \cite[section 4.6]{Basar:1999} and \cite{FKB:2012} for background). Hence, consider the profits being given by 
\begin{align}
J_{1}(t)&=\frac{1}{2}u^{T}(t)H_{1}u(t)+h_{1}^{T}u(t)+c_{1}\,, \label{eq:J1_20240311} \\
J_{2}(t)&=\frac{1}{2}u^{T}(t)H_{2}u(t)+h_{2}^{T}u(t)+c_{2}\,, \label{eq:J2_20240311} \\
J_{2}(t)&=\frac{1}{2}u^{T}(t)H_{3}u(t)+h_{3}^{T}u(t)+c_{3}\,, \label{eq:J3_20240311} \\
J_{4}(t)&=\frac{1}{2}u^{T}(t)H_{4}u(t)+h_{4}^{T}u(t)+c_{4}\,, \label{eq:J4_20240311}
\end{align}
with vector of prices $u(t):=[u_{1}(t) \ u_{2}(t) \ u_{3}(t) \ u_{4}(t)]^{T} \in \mathbb{R}^{4}$ and their corresponding parameters 
\begin{align}
&H_{1}= \frac{1}{(R_{2}R_{3}R_{4}+R_{1}R_{3}R_{4}+R_{1}R_{2}R_{4}+R_{1}R_{2}R_{3})}\times \nonumber \\
&\times\begin{bmatrix} -2(R_{3}R_{4}+R_{2}R_{4}+R_{2}R_{3}) &R_{3}R_{4} & R_{2}R_{4} & R_{2}R_{3} \\ 
R_{3}R_{4} & 0 & 0 & 0 \\
R_{2}R_{4} & 0 & 0 & 0 \\
R_{2}R_{3} & 0 & 0 & 0 \end{bmatrix} \,, \\ 
&h_{1} = \frac{1}{(R_{2}R_{3}R_{4}+R_{1}R_{3}R_{4}+R_{1}R_{2}R_{4}+R_{1}R_{2}R_{3})} \times \nonumber \\
&\times\begin{bmatrix} m_{1}(R_{3}R_{4}+R_{2}R_{4}+R_{2}R_{3})+S_{d}R_{2}R_{3}R_{4}  \\ 
-m_{1}R_{3}R_{4}  \\
-m_{1}R_{2}R_{4}  \\
-m_{1}R_{2}R_{3}  \end{bmatrix}\,, \\
&c_{1} = -\frac{m_{1}S_{d}R_{2}R_{3}R_{4}}{(R_{2}R_{3}R_{4}+R_{1}R_{3}R_{4}+R_{1}R_{2}R_{4}+R_{1}R_{2}R_{3})}\,,
\end{align}

\begin{align}
&H_{2}= \frac{1}{(R_{2}R_{3}R_{4}+R_{1}R_{3}R_{4}+R_{1}R_{2}R_{4}+R_{1}R_{2}R_{3})}\times \nonumber \\
&\times\begin{bmatrix} 0 & R_{3}R_{4} & 0 & 0 \\ 
 R_{3}R_{4} & -2(R_{3}R_{4}+R_{1}R_{4}+R_{1}R_{3}) & R_{1}R_{4} & R_{1}R_{3} \\
0 & R_{1}R_{4} & 0 & 0 \\
0 & R_{1}R_{3} & 0 & 0 \end{bmatrix} \,, \\ 
&h_{2} = \frac{1}{(R_{2}R_{3}R_{4}+R_{1}R_{3}R_{4}+R_{1}R_{2}R_{4}+R_{1}R_{2}R_{3})} \times \nonumber \\
&\times\begin{bmatrix} -m_{2}R_{3}R_{4} \\ 
m_{2}(R_{3}R_{4}+R_{1}R_{4}+R_{1}R_{3})+S_{d}R_{1}R_{3}R_{4}  \\
-m_{2}R_{1}R_{4}  \\
-m_{2}R_{1}R_{3}  \end{bmatrix}\,, \\
&c_{2} = -\frac{m_{2}S_{d}R_{1}R_{3}R_{4}}{(R_{2}R_{3}R_{4}+R_{1}R_{3}R_{4}+R_{1}R_{2}R_{4}+R_{1}R_{2}R_{3})}\,,
\end{align}

\begin{align}
&H_{3}= \frac{1}{(R_{2}R_{3}R_{4}+R_{1}R_{3}R_{4}+R_{1}R_{2}R_{4}+R_{1}R_{2}R_{3})}\times \nonumber \\
&\times \begin{bmatrix} 0 & 0 & R_{2}R_{4} & 0 \\ 
 0 & 0 &  R_{1}R_{4} & 0 \\
R_{2}R_{4} & R_{1}R_{4} & -2(R_{2}R_{4}+R_{1}R_{4}+R_{1}R_{2}) & R_{1}R_{2} \\
0 & 0 & R_{1}R_{2} & 0 \end{bmatrix} \,, \\ 
&h_{3} = \frac{1}{(R_{2}R_{3}R_{4}+R_{1}R_{3}R_{4}+R_{1}R_{2}R_{4}+R_{1}R_{2}R_{3})}\times \nonumber \\
&\times\begin{bmatrix} -m_{3}R_{2}R_{4} \\ 
-m_{3}R_{1}R_{4}  \\
m_{3}(R_{2}R_{4}+R_{1}R_{4}+R_{1}R_{2})+S_{d}R_{1}R_{2}R_{4}  \\
-m_{3}R_{1}R_{2}  \end{bmatrix}\,, \\
&c_{3} = -\frac{m_{3}S_{d}R_{1}R_{2}R_{4}}{(R_{2}R_{3}R_{4}+R_{1}R_{3}R_{4}+R_{1}R_{2}R_{4}+R_{1}R_{2}R_{3})}\,, \\
\end{align}

\begin{align}
&H_{4}= \frac{1}{(R_{2}R_{3}R_{4}+R_{1}R_{3}R_{4}+R_{1}R_{2}R_{4}+R_{1}R_{2}R_{3})}\times \nonumber \\
&\times \begin{bmatrix} 0 & 0 & 0 & R_{2}R_{3} \\ 
 0 & 0 &  0 & R_{1}R_{3} \\
0 & 0 & 0 & R_{1}R_{2} \\
R_{2}R_{3} & R_{1}R_{3} & R_{1}R_{2} & -2(R_{2}R_{3}+R_{1}R_{3}+R_{1}R_{2})\end{bmatrix} \,, \\ 
&h_{4} = \frac{1}{(R_{2}R_{3}R_{4}+R_{1}R_{3}R_{4}+R_{1}R_{2}R_{4}+R_{1}R_{2}R_{3})}\times \nonumber \\
&\times\begin{bmatrix} -m_{4}R_{2}R_{3} \\ 
-m_{4}R_{1}R_{3}  \\
-m_{4}R_{1}R_{2} \\
m_{4}(R_{2}R_{3}+R_{1}R_{3}+R_{1}R_{2})+S_{d}R_{1}R_{2}R_{3}  \end{bmatrix}\,, \\
&c_{4} = -\frac{m_{4}S_{d}R_{1}R_{2}R_{3}}{(R_{2}R_{3}R_{4}+R_{1}R_{3}R_{4}+R_{1}R_{2}R_{4}+R_{1}R_{2}R_{3})}\,.
\end{align}

\begin{figure*}
\begin{adjustbox}{addcode={\begin{minipage}{\width}}
{\begin{align}
&H= \frac{1}{(R_{2}R_{3}R_{4}+R_{1}R_{3}R_{4}+R_{1}R_{2}R_{4}+R_{1}R_{2}R_{3})}\times \nonumber \\
&\times\begin{bmatrix} -2(R_{3}R_{4}+R_{2}R_{4}+R_{2}R_{3}) &R_{3}R_{4} & R_{2}R_{4} & R_{2}R_{3} \\ 
 R_{3}R_{4} & -2(R_{3}R_{4}+R_{1}R_{4}+R_{1}R_{3}) & R_{1}R_{4} & R_{1}R_{3} \\
R_{2}R_{4} & R_{1}R_{4} & -2(R_{2}R_{4}+R_{1}R_{4}+R_{1}R_{2}) & R_{1}R_{2} \\
R_{2}R_{3} & R_{1}R_{3} & R_{1}R_{2} & -2(R_{2}R_{3}+R_{1}R_{3}+R_{1}R_{2})\end{bmatrix} \,, \label{eq:H_20240311} \\
&h=\frac{1}{(R_{2}R_{3}R_{4}+R_{1}R_{3}R_{4}+R_{1}R_{2}R_{4}+R_{1}R_{2}R_{3})}\begin{bmatrix} m_{1}(R_{3}R_{4}+R_{2}R_{4}+R_{2}R_{3})+S_{d}R_{2}R_{3}R_{4}  \\ 
m_{2}(R_{3}R_{4}+R_{1}R_{4}+R_{1}R_{3})+S_{d}R_{1}R_{3}R_{4}  \\
m_{3}(R_{2}R_{4}+R_{1}R_{4}+R_{1}R_{2})+S_{d}R_{1}R_{2}R_{4}  \\
m_{4}(R_{2}R_{3}+R_{1}R_{3}+R_{1}R_{2})+S_{d}R_{1}R_{2}R_{3}  \end{bmatrix}\,. \label{eq:h_20240311}
\end{align}
\end{minipage}},rotate=90,center}
\end{adjustbox}
\end{figure*}

The constants $m_{1}$, $m_{2}$, $m_{3}$ and $m_{4}$ represent the marginal costs, $S_{d}$ is the total consumer demand, and $R_{1}$, $R_{2}$, $R_{3}$ and $R_{4}$ describe the resistance that consumers have toward buying a given product. This reluctance
 may be due to the quality or the brand image considerations---the most desirable products have lowest resistance. The payoff functions (\ref{eq:J1_20240311})--(\ref{eq:J4_20240311}) are in the form (\ref{eq:Ji_v1}). Therefore, the Nash equilibrium $\theta^{\ast} = [\theta_{1}^{\ast}\,, \theta_{2}^{\ast}\,, \theta_{3}^{\ast}\,, \theta_{4}^{\ast}]^{T}$ satisfies Assumption~\ref{assum:Nash} with $H$ in (\ref{eq:H_20240311}) and $h$ in (\ref{eq:h_20240311}). 

In addition, according to Assumption~\ref{assum:SDD}, the matrix $H$ in (\ref{eq:H_20240311}) is strictly diagonally dominant, thus, the Nash equilibrium $\theta^{\ast}$ exists and this is unique since strictly diagonally dominant matrices are nonsingular by Levy-Desplanques Theorem \cite{T:1949}. Moreover, $H^{i}_{ij}=H^{j}_{ji}$, then $H$ in (\ref{eq:H_20240311}) is a negative definite symmetric matrix and the nooncoperative game belongs to a class of games known as \textit{potential games} \cite{MS:1996}.

We have simulated the Lie-bracket NES strategy for the static noncooperative game above with four firms in an oligopoly market structure following the design procedure presented in this paper. The parameters of the plant are the same as those in \cite{FKB:2012}: with initial conditions $\hat{\theta}_{1}(0)=52$, $\hat{\theta}_{2}(0)=40.93$, $\hat{\theta}_{3}(0)=33.5$, $\hat{\theta}_{4}(0)=35.09$; $S_{d} = 100$, $R_{1}=0.15$, $R_{2}=0.30$, $R_{3}=0.60$, $R_{4}=1$, $m_{1}=30$, $m_{2}=30$, $m_{3}=25$, $m_{4}=20$, which, according to (\ref{eq:NE}), yield the unique Nash equilibrium 
\begin{align} 
\theta^{*}&=\begin{bmatrix}42.8818 & 40.9300 & 37.8363 & 35.0874\end{bmatrix}^{T}\,, \label{eq:theta_NE} \\
J^{*}&=\begin{bmatrix}524.0208 & 293.4217 & 238.4846 & 209.6584\end{bmatrix}^{T}\,. \label{eq:J_NE}
\end{align}
Moreover, the control parameters are chosen as: $\alpha_{1}=\alpha_{2}=\alpha_{3}=\alpha_{4}=0.05$, $k_{1}=6$, $k_{2}=18$, $k_{3}=10$, $k_{4}=24$, $\omega_{1}=30$, $\omega_{2}=24$, $\omega_{3}=44$, $\omega_{4}=36$.

\begin{figure}[h!]
	\centering
        \centering
        \includegraphics[width=0.7\linewidth]{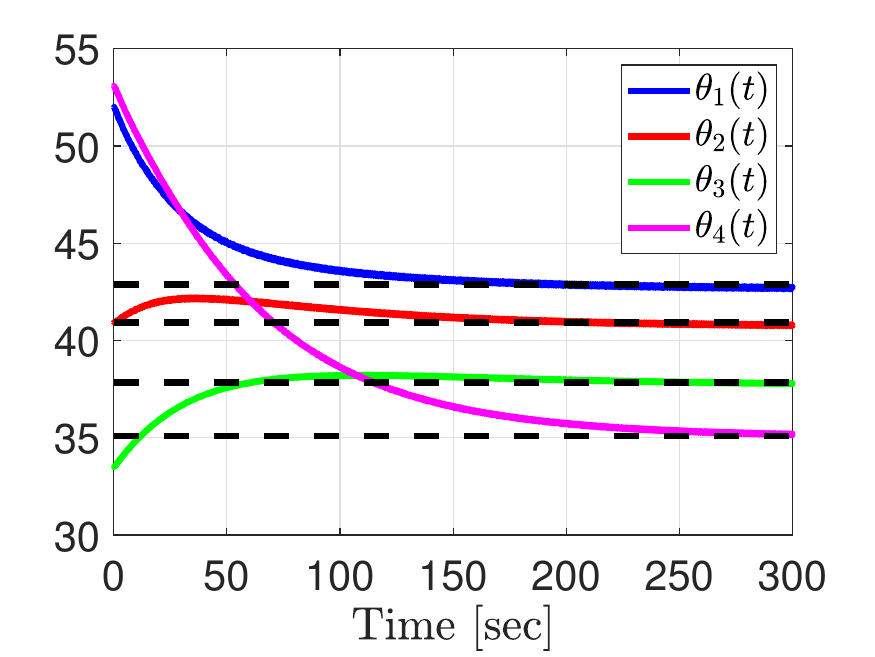} 
        \caption{Players' actions, $\theta_i(t)$ .}
        \label{fig:theta}
\end{figure}
\begin{figure}[h!]
        \centering
        \includegraphics[width=0.7\linewidth]{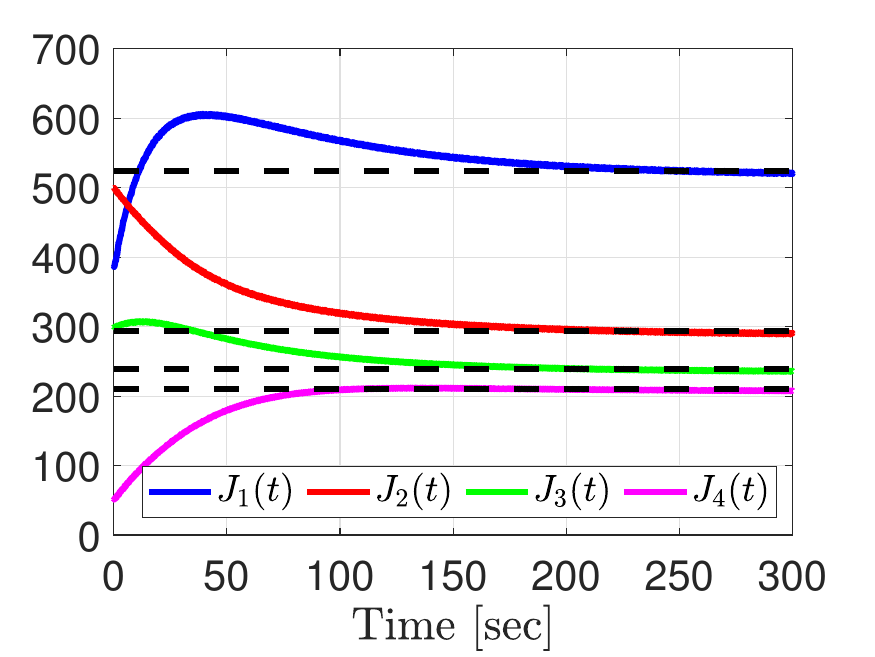} 
        \caption{Payoff functions, $J_i(t)$.}
        \label{fig:J}
\end{figure}

In order to achieve the Nash equilibrium (\ref{eq:theta_NE}) and (\ref{eq:J_NE}), without requiring detailed modeling information, Players $P_1$, $P_2$, $P_3$, and $P_4$ implement the proposed Lie-Bracket NES strategy to determine their optimal actions. Figures~\ref{fig:theta} and \ref{fig:J} depict the time evolution of the proposed Lie-Bracket NES approach, evaluating the coordinated efforts to drive the system toward the Nash equilibrium, as shown in Figure~\ref{fig:theta}. Additionally, Figure~\ref{fig:J} illustrates how, within the described oligopoly market structure, the four firms—without sharing any information—successfully maximize their profits $J_i(t)$ by employing the proposed strategy to determine the price $\theta_i(t)$ of their product in the static noncooperative game.

\section{Conclusion} \label{sec7}

This paper has introduced a novel Nash equilibrium seeking method with bounded update rates. Locally stable convergence to Nash equilibrium in quadratic noncooperative games was rigorously guaranteed. The stability analysis was conducted using Lie-Bracket techniques and Lyapunov's direct method. Moreover, our result quantified the size of the ultimate small residual sets around the Nash equilibrium, being independent of the amplitudes of the perturbation-probing signals, unlike in classical extremum seeking. A numerical example validated the effectiveness of the proposed approach by considering an oligopoly market structure, with all firms implementing the proposed distributed Lie-bracket Nash seeking policy to achieve the Nash equilibrium without sharing information among them. Ultimately, the proposed strategy enables all firms to improve their profits such that the action updates occur independently of each other. 

Future work could examine cases where players use different algorithms, some exhibit stubborn behavior as in \cite{FKB:2012}, and deception, as characterized in \cite{Tang:2024}, or the framework is extended to non-quadratic cost functions and bounded rationality. The proposed update law is interesting, not only because it is bounded but also because it extends to unicycles---which could be explored for a game with mobile robots
\cite{Zha:2007,DSEJ:2013}, but using instead angular velocity tuning and constant (bounded) forward velocity \cite{Scheinker:2017}.

\section*{Acknowledgment}

This study was financed in part by the Coordena{\c c}{\~a}o de Aperfei{\c c}oamento de Pessoal de N{\'i}vel Superior – Brasil (CAPES) – Finance Code 001. The authors also acknowledge the Brazilian Funding Agencies Conselho Nacional de Desenvolvimento Cient{\'i}fico e Tecnol{\'o}gico (CNPq) and Funda{\c c}{\~a}o de Amparo {\`a} Pesquisa do Estado do Rio de Janeiro (FAPERJ).

\section*{Appendix}

\subsection*{Averaging of Nonholonomic Systems \cite{GL:1993}}

In this section, we recall the averaging technique introduced in \cite{GL:1993} to compute asymptotic behaviour of a nonholonomic system under application of some highly oscillatory inputs.

Consider the following nonholonomic system: 
\begin{equation}
    \dot{x} = \sum_{i=1}^m b_i(x) u_i^\epsilon\,, \label{eq:appendix_1}
\end{equation}
with control input of the form
\begin{equation}
    u_i^\epsilon = \bar{u}_i(t) + \frac{1}{\sqrt{\epsilon}} \tilde{u}_i(t, \theta)\,, \label{eq:appendix_2}
\end{equation}
where $\theta = t / \epsilon$, and the function $\tilde{u}_i$ is periodic in $\theta$ with period $T$, and has zero average
\begin{equation}
    \int_0^{T} \tilde{u}_i(t, \theta) \, d\theta = 0. \label{eq:appendix_3}
\end{equation}

We can identify $\theta$ as the \textit{fast} variable, and $(x, t)$ the \textit{slow} variables. We can, for small enough $\epsilon$, replace the  system (\ref{eq:appendix_1})--(\ref{eq:appendix_3}) with an ``average'' system for which the fast variable is eliminated. Moreover, the resulting average system serves as a good approximation of system (\ref{eq:appendix_1}).

\begin{theorem}
For sufficiently small $\epsilon$, the trajectory of system (\ref{eq:appendix_1}) is bounded by the solution of the following system:
\begin{equation}
    \dot{z} = \sum_{i=1}^m b_i(z) \bar{u}_i + \frac{1}{T} \sum_{i<j} [b_i, b_j] \nu_{i,j}, \quad z(0) = x(0) = x_0 \label{eq:appendix_4}
\end{equation}
in the sense that
\begin{equation}
    \|x - z\| \leq \mathcal{O}(\epsilon), \label{eq:appendix_5}
\end{equation}
where
\begin{equation}
    \nu_{i,j} = \int_0^{T} \int_0^{\theta} \tilde{u}_i(t, \tau) \tilde{u}_j(t, \theta) d\tau d\theta \label{eq:appendix_6}
\end{equation}
and $\mathcal{O}(\epsilon)$ is a parameter which tends to zero as $\epsilon \to 0$.
\end{theorem}

In other words, system (\ref{eq:appendix_4}) defines the asymptotic behaviour of system (\ref{eq:appendix_1})--(\ref{eq:appendix_3}), as $\epsilon \to 0$.


\begin{thebibliography}{28}
\providecommand{\natexlab}[1]{#1}
\providecommand{\url}[1]{\texttt{#1}}
\expandafter\ifx\csname urlstyle\endcsname\relax
  \providecommand{\doi}[1]{doi: #1}\else
  \providecommand{\doi}{doi: \begingroup \urlstyle{rm}\Url}\fi

\bibitem[Ba{\c s}ar and Olsder(1999)]{Basar:1999}
T.~Ba{\c s}ar and G.~J. Olsder.
\newblock \emph{Dynamic Noncooperative Game Theory}.
\newblock SIAM Series in Classics in Applied Mathematics. SIAM, Philadelphia,
  1999.

\bibitem[D{\" u}rr et~al.(2013)D{\" u}rr, Stankovi{\' c}, Ebenbauer, and
  Johansson]{DSEJ:2013}
H.-B. D{\" u}rr, M.~S. Stankovi{\' c}, C.~Ebenbauer, and K.~H. Johansson.
\newblock Lie bracket approximation of extremum seeking systems.
\newblock \emph{Automatica}, 49:\penalty0 1538--1552, 2013.

\bibitem[D{\" u}rr et~al.(2015)D{\" u}rr, Krsti{\' c}, Scheinker, and
  Ebenbauer]{DKSE:2015}
H.-B. D{\" u}rr, M.~Krsti{\' c}, A.~Scheinker, and C.~Ebenbauer.
\newblock Singularly perturbed lie bracket approximation.
\newblock \emph{IEEE Transactions on Automatic Control}, 60:\penalty0
  3287--3292, 2015.

\bibitem[Frihauf et~al.(2012)Frihauf, Krsti{\' c}, and Ba{\c s}ar]{FKB:2012}
P.~Frihauf, M.~Krsti{\' c}, and T.~Ba{\c s}ar.
\newblock Nash equilibrium seeking in noncooperative games.
\newblock \emph{IEEE Transactions on Automatic Control}, 57:\penalty0
  1192--1207, 2012.

\bibitem[Gurvits and Li(1993)]{GL:1993}
L.~Gurvits and Z.~X. Li.
\newblock Smooth time-periodic feedback solutions for nonholonomic motion
  planning.
\newblock In Z.~Li and J.~F. Canny, editors, \emph{Nonholonomic Motion
  Planning}, volume~1, pages 53--90. Springer Science, Business Media New York,
  New York, NY, USA, 1st edition, 1993.

\bibitem[Horn and Johnson(1985)]{HJ:1985}
R.~A. Horn and C.~R. Johnson.
\newblock \emph{Matrix Analysis}.
\newblock Cambridge University Press, 1985.

\bibitem[Khalil(2002)]{K:2002}
H.~K. Khalil.
\newblock \emph{Nonlinear Systems}.
\newblock Prentice Hall, Upper Saddle River, NJ, 3rd edition, 2002.

\bibitem[Krsti{\' c} and Wang(2000)]{KW:2000}
M.~Krsti{\' c} and H.-H. Wang.
\newblock Stability of extremum seeking feedback for general nonlinear dynamic
  systems.
\newblock \emph{Automatica}, 36:\penalty0 595--601, 2000.

\bibitem[Mardsden and Tromba(2003)]{MT:2003}
J.~E. Mardsden and A.~J. Tromba.
\newblock \emph{Vector Calculus}.
\newblock W.~H. Freeman and Company, New York, NY, USA, 5th edition, 2003.

\bibitem[Monderer and Shapley(1996)]{MS:1996}
D.~Monderer and L.~S. Shapley.
\newblock Potential games.
\newblock \emph{Games and Economic Behavior}, 14:\penalty0 124--143, 1996.

\bibitem[Nash(1951)]{N:1951}
J.~F. Nash.
\newblock Noncooperative games.
\newblock \emph{Annals of Mathematics}, 54:\penalty0 286--295, 1951.

\bibitem[Oliveira and Krstic(2022)]{SIAM:2022}
T.~R. Oliveira and M.~Krstic.
\newblock \emph{Extremum Seeking through Delays and PDEs}.
\newblock SIAM, Philadelphia, 2022.

\bibitem[Oliveira et~al.(2021)Oliveira, Rodrigues, Krsti{\' c}, and Ba{\c
  s}ar]{JOTA:2020}
T.~R. Oliveira, V.~H.~P. Rodrigues, M.~Krsti{\' c}, and T.~Ba{\c s}ar.
\newblock Nash equilibrium seeking in quadratic noncooperative games under two
  delayed information-sharing schemes.
\newblock \emph{Journal of Optimization Theory and Applications}, 191:\penalty0
  700--735, 2021.

\bibitem[Oliveira et~al.(2024)Oliveira, Krsti{\' c}, and Ba{\c s}ar]{CSm:2024}
T.~R. Oliveira, M.~Krsti{\' c}, and T.~Ba{\c s}ar.
\newblock Extremum and nash equilibrium seeking with delays and pdes: designs
  \& applications.
\newblock \emph{Arxiv}, 2024.
\newblock \url{https://doi.org/10.48550/arXiv.2411.13234}.

\bibitem[Poveda et~al.(2023)Poveda, Krsti{\' c}, and Ba{\c s}ar]{Poveda:2023}
J.~I. Poveda, M.~Krsti{\' c}, and T.~Ba{\c s}ar.
\newblock Fixed-time nash equilibrium seeking in time-varying networks.
\newblock \emph{IEEE Trans. Automat. Contr.}, 68:\penalty0 1954--1969, 2023.

\bibitem[Rodrigues et~al.(2024{\natexlab{a}})Rodrigues, Oliveira, Krsti{\' c},
  and Ba{\c s}ar]{CDC:2024}
V.~H.~P. Rodrigues, T.~R. Oliveira, M.~Krsti{\' c}, and T.~Ba{\c s}ar.
\newblock Nash equilibrium seeking for noncooperative duopoly games via
  event-triggered control.
\newblock \emph{Arxiv}, 2024{\natexlab{a}}.
\newblock \url{https://doi.org/10.48550/arXiv.2404.07287}.

\bibitem[Rodrigues et~al.(2024{\natexlab{b}})Rodrigues, Oliveira, Krsti{\' c},
  and Ba{\c s}ar]{VSS:2024}
V.~H.~P. Rodrigues, T.~R. Oliveira, M.~Krsti{\' c}, and T.~Ba{\c s}ar.
\newblock Sliding-mode nash equilibrium seeking for a quadratic duopoly game.
\newblock \emph{Arxiv}, 2024{\natexlab{b}}.
\newblock \url{https://doi.org/10.48550/arXiv.2405.15762}.

\bibitem[Scheinker(2017)]{Scheinker:2017}
A.~Scheinker.
\newblock Bounded extremum seeking for angular velocity actuated control of
  nonholonomic unicycle.
\newblock \emph{Optimal Control Applications \& Methods}, 38:\penalty0
  575--585, 2017.

\bibitem[Scheinker(2024)]{S:2024}
A.~Scheinker.
\newblock 100 years of extremum seeking: a survey.
\newblock \emph{Automatica}, 161\penalty0 (111481):\penalty0 1--39, 2024.

\bibitem[Scheinker and Krsti{\' c}(2014)]{SK_SCL:2014}
A.~Scheinker and M.~Krsti{\' c}.
\newblock Extremum seeking with bounded update rates.
\newblock \emph{Systems \& Control Letters}, 63:\penalty0 25--31, 2014.

\bibitem[Scheinker and Krsti{\' c}(2017)]{SK:2017}
A.~Scheinker and M.~Krsti{\' c}.
\newblock \emph{Model-Free Stabilization by Extremum Seeking}.
\newblock Springer, 1st edition, 2017.

\bibitem[Scheinker and Scheinker(2016)]{SS:2016}
A.~Scheinker and D.~Scheinker.
\newblock Bounded extremum seeking with discontinuous dithers.
\newblock \emph{Automatica}, 69:\penalty0 250--257, 2016.

\bibitem[Scheinker et~al.(2019)Scheinker, Bohler, Tomin, Kammering, Zagorodnov,
  Schlarb, Scholz, Beutner, and Decking]{SBTKZSSBD:2019}
A.~Scheinker, D.~Bohler, S.~Tomin, R.~Kammering, I.~Zagorodnov, H.~Schlarb,
  M.~Scholz, B.~Beutner, and W.~Decking.
\newblock Model-independent tuning for maximizing free electron laser pulse
  energy.
\newblock \emph{Physical Review Accelerators and Beams}, 22:\penalty0 082802,
  2019.

\bibitem[Scheinker et~al.(2020)Scheinker, Hirlaender, Velotti, Gessner, Porta,
  Kain, Goddard, and Ramjiawan]{SHVGPDK:2020}
A.~Scheinker, S.~Hirlaender, F.~M. Velotti, S.~Gessner, G.~Z.~D. Porta,
  V.~Kain, B.~Goddard, and R.~Ramjiawan.
\newblock Online multi-objective particle accelerator optimization of the awake
  electron beam line for simultaneous emittance and orbit control.
\newblock \emph{AIP Advances}, 10:\penalty0 055320, 2020.

\bibitem[Scheinker et~al.(2022)Scheinker, Huang, and Taylor]{SHT:2021}
A.~Scheinker, E.-C. Huang, and C.~Taylor.
\newblock Extremum seeking-based control system for particle accelerator beam
  loss minimization.
\newblock \emph{IEEE Transactions on Control Systems Technology}, 30:\penalty0
  2261--2268, 2022.

\bibitem[Tang et~al.(2024)Tang, Javed, Chen, Krsti{\' c}, and
  Poveda]{Tang:2024}
M.~Tang, U.~Javed, X.~Chen, M.~Krsti{\' c}, and J.~I. Poveda.
\newblock Deception in nash equilibrium seeking.
\newblock \emph{Arxiv}, 2024.
\newblock \url{https://doi.org/10.48550/arXiv.2407.05168}.

\bibitem[Taussky(1949)]{T:1949}
O.~Taussky.
\newblock A recurring theorem on determinants.
\newblock \emph{The American Mathematical Monthly}, 56:\penalty0 672--676,
  1949.

\bibitem[Zhang et~al.(2007)Zhang, Arnold, Ghods, Siranosian, and
  Krstic]{Zha:2007}
C.~Zhang, D.~Arnold, N.~Ghods, A.~Siranosian, and M.~Krstic.
\newblock Source seeking with nonholonomic unicycle without position
  measurement and with tuning of forward velocity.
\newblock \emph{Systems and Control Letters}, 56:\penalty0 245--252, 2007.

\end{thebibliography}
\end{document}